\documentclass[12pt,english]{article}
\usepackage[LGR,T1]{fontenc}
\usepackage[latin9]{inputenc}
\usepackage{geometry}
\usepackage{color}
\usepackage{amsmath}
\usepackage{amsthm}
\usepackage{amssymb}

\makeatletter

\DeclareRobustCommand{\greektext}{%
  \fontencoding{LGR}\selectfont\def\encodingdefault{LGR}}
\DeclareRobustCommand{\textgreek}[1]{\leavevmode{\greektext #1}}
\ProvideTextCommand{\~}{LGR}[1]{\char126#1}

\numberwithin{equation}{section}
\theoremstyle{plain}
\newtheorem{thm}{\protect\theoremname}[section]
\theoremstyle{plain}
\newtheorem{lem}[thm]{\protect\lemmaname}
\theoremstyle{plain}
\newtheorem{cor}[thm]{\protect\corollaryname}
\theoremstyle{remark}
\newtheorem{rem}[thm]{\protect\remarkname}
\theoremstyle{definition}
\newtheorem{example}[thm]{\protect\examplename}

\usepackage{graphicx}
\usepackage{float}
\usepackage{comment}

\usepackage{hyperref}
\hypersetup{
     colorlinks   = true,
     citecolor    = blue,
     linkcolor    = blue
}

\date{}

\makeatother

\usepackage{babel}
\providecommand{\corollaryname}{Corollary}
\providecommand{\examplename}{Example}
\providecommand{\lemmaname}{Lemma}
\providecommand{\remarkname}{Remark}
\providecommand{\theoremname}{Theorem}

\begin{document}
\title{A representation for the Expected Signature of Brownian motion up
to the first exit time of the planar unit disc }
\author{Horatio Boedihardjo, Lin He and Lisa Wang \thanks{\textcolor{black}{All authors gratefully acknowledges the support
from University of Warwick's URSS (Undergraduate Research Support
Scheme).} Please direct all enquiries to the undergraduate research
supervisor, Horatio Boedihardjo, via horatio.boedihardjo@warwick.ac.uk.\textcolor{black}{{} }}}
\maketitle
\begin{abstract}
The signature of a sample path is a formal series of iterated integrals
along the path. The expected signature of a stochastic process gives
a summary of the process that is especially useful for studying stochastic
differential equations driven by the process. Lyons-Ni derived a partial
differential equation for the expected signature of Brownian motion,
starting at a point $z$ in a bounded domain, until it hits to boundary
of the domain. We focus on the domain of planar unit disc centred
at 0. Motivated by recently found explicit formulae for some terms
in the expected signature of this process in terms of Bessel functions,
we derive a tensor series representation for this expected signature,
coming from from studying Lyons-Ni's PDE. Although the representation
is rather involved, it simplifies significantly to give a formula
for the polynomial leading order term in each tensor component of
the expected signature.
\end{abstract}

\section{Introduction }

Given a bounded variation path $x:[0,T]\rightarrow\mathbb{R}^{d}$,
the signature of $x$ is defined as the tensor series 
\begin{equation}
S(x)_{0,T}=(1,\int_{0}^{T}\mathrm{d}x_{t_{1}},\cdots,\int_{0<t_{1}<\cdots<t_{n}<T}\mathrm{d}x_{t_{1}}\otimes\cdots\otimes\mathrm{d}x_{t_{n}},\cdots).\label{eq:Signature}
\end{equation}
The idea of signature can be extended so that it is defined on ``rough
paths'' \cite{Lyo98}, which may not have bounded variation.In the
case when $x$ is a multi-dimensional Brownian motion, the Stratonovich
signature of $x$ is defined by (\ref{eq:Signature}) with the integration
being defined in the Stratonovich sense. 

The signature appears naturally when applying Picard's iteration to
solve differential equation of the form 
\[
\mathrm{d}Y_{t}=A(Y_{t})\mathrm{d}x_{t},
\]
where $A$ is a linear function. In fact, any solution $Y_{T}$ can
be expressed as a composition of a linear map $F_{A}$, which depends
on $A$, of $S(x)_{0,T}$ and $Y_{0}$. Therefore, $S(x)_{0,T}$ fully
captures the effect of $x$ on $Y_{T}$. As a result, the signature
of $x$ is considered, in the study of path-driven differential equations,
to be a useful summary of the path $x$, in a similar way as exponential
function is useful in the study of ordinary differential equations.
Signatures uniquely determine paths up to ``tree-like equivalence''
(\cite{Uniqueness1,Uniqueness2,YamThesis}) and various schemes have
been developed to reconstruct a path from its signature (\cite{Inversion1,Inversion2,Inversion3,Inversion4,Inversion5,Inversion6,Inversion7,Inversion8}). 

Likewise, when $x$ is a stochastic process, the probability distribution
of $S(x)_{0,T}$ would be a useful summary of the distribution of
$x$ as far as stochastic differential equations is concerned. The
algebraic properties of signature make it an even more attractive
tool. For example, Chevyrev-Lyons \cite{ChevyrevLyons} showed that
the expected signature of a stochastic process $x$ uniquely determines
the probability distribution of the signature of $x$, assuming certain
decay condition of the expected signature of $x$. This relies heavily
on the algebraic structure of the signature, because in general the
expected value of a random variable does not determine the distribution
of the random variable. Further recent applications of signature can
be found in references (\cite{Application1,Application2,Application3,Application4}).

The expected signature of the following processes have been computed,
listed in broadly chronological order: 

1. Brownian motion up to a deterministic time: Fawcett \cite{Fawcett}
and, using a different method, in Lyons-Victoir \cite{LyonsVictoir};

2. Fractional Brownian motions and other Gaussian processes in \cite{BaudoinCoutin}
for Hurst parameter $H>\frac{1}{2}$ and in \cite{FerrucciCass} for
$H>\frac{1}{4}$. See also the work of \cite{Passeggeri20} on weak
convergence rate;

3. SLE curves \cite{Werness12} (first three terms), \cite{UniquenessSimple}
(fourth term);

4. PDE representation for Brownian motion up to the exit time of a
bounded domain \cite{LyonsNi};

5. PDE representation for Diffusions \cite{NiHaoThesis,LiNiZhu};

6. Lévy processes in \cite{FrizShekhar}. 

This paper follows on naturally from item 4 \cite{LyonsNi}, which
derived a PDE for the expected signature for Brownian motion up to
the exit time of a domain. The paper \cite{LyonsNi} also derived
the first four terms of the expected signature was computed and produces
a recursive relation expressing the higher degree terms in terms of
lower degree terms. As far as we know, there has been no explicit
formula to solve the recurrence relation. In a more recent paper \cite{ExitROC},
an explicit formula for a \emph{projection} of the expected signature
has been found when the domain is a two dimensional unit disc, and
a multidimensional version can be found in \cite{LiNi}. As the projection
involves a sum of infinitely many terms, these works have given hope
that perhaps a reasonably tractable formula for the expected signature
of Brownian motion up to the exit time of a planar unit disc may exist.
The purpose of this work is to explore whether such formula indeed
exists, or equivalently, whether the PDE derived by Lyons and Ni \cite{LyonsNi}
can be solved in some way. Failing that, it would be interesting to
know which terms in the expected signature of Brownian motion up to
exit time of planar unit disc can be computed in a tidy manner using
the ideas from \cite{ExitROC,LiNi}. 

The main result of this paper is a tensor series representation of
the expected signature of Brownian motion up to the first exit time
of the planar unit disc (see Theorem \ref{thm:MainTheorem}). The
formula is unfortunately rather complicated, but the formula does
imply a nice expression for the leading order term in the expected
signature for each tensor degree (see Corollary \ref{cor:MainCorollary}). 

The plan for the paper is as follows. In Section 2, we will introduce
some notations. In Section 3, we will give the process of deriving
an expression for the solution of Lyons-Ni's PDE. 

\section{Notation}

Let $T((\mathbb{R}^{2}))$ be set of formal series of tensors over
$\mathbb{R}^{2}$, or more precisely, $T((\mathbb{R}^{2}))$ is the
set of all sequences 
\[
T((\mathbb{R}^{2}))=\{(a_{0},a_{1},a_{2},\ldots):a_{i}\in(\mathbb{R}^{2})^{\otimes i}\}.
\]
The map $\rho_{i}:T((\mathbb{R}^{2}))\rightarrow(\mathbb{R}^{2})^{\otimes i}$
is the projection of $T((\mathbb{R}^{2}))$ onto $(\mathbb{R}^{2})^{\otimes i}$,
sending each sequence $(a_{0},a_{1},\ldots)$ to the $i$-th entry
$a_{i}$. We let $\mathbf{1}$ denote multiplicative identity element
$(1,0,0,\cdots)$.

Let $\mathbb{D}$ be the two dimensional unit disc $\mathbb{D}=\{(x_{1},y)^{T}\in\mathbb{R}^{2}:x^{2}+y^{2}<1\}$.
Let $B^{z}$ be a Brownian motion starting from $z\in\mathbb{D}$
and let $\tau_{\mathbb{D}}$ be the first time $B^{z}$ exit $\mathbb{D}$. 

The signature of $B^{z}$ up to $\tau_{\mathbb{D}}$ is an element
of $T((\mathbb{R}^{2}))$ defined by 
\[
S(B^{z})_{0,\tau_{\mathbb{D}}}=(1,\int_{0}^{\tau_{\mathbb{D}}}\mathrm{d}B_{t_{1}}^{z},\ldots,\int_{0<t_{1}<\cdots<t_{n}<\tau_{\mathbb{D}}}\circ\mathrm{d}B_{t_{1}}^{z}\otimes\cdots\otimes\circ\mathrm{d}B_{t_{n}}^{z},\ldots),
\]
with the integration being defined in a Stratonovich sense. The expected
signature of $B^{z}$ up to the first exit time of $\mathbb{D}$ is
an element $\Phi_{\mathbb{D}}(z)$ of $T((\mathbb{R}^{2}))$ whose
$n$-th term is 
\[
\rho_{n}\big[\Phi_{\mathbb{D}}(z)\big]=\mathbb{E}\left[\int_{0<t_{1}<\cdots<t_{n}<\tau_{\mathbb{D}}}\circ\mathrm{d}B_{t_{1}}^{z}\otimes\cdots\otimes\circ\mathrm{d}B_{t_{n}}^{z}\right]
\]
and is first studied in \cite{LyonsNi}. Let $\{e_{1},e_{2}\}$ be
the standard basis of $\mathbb{R}^{2}$, that is 
\[
e_{1}=\left(\begin{array}{c}
1\\
0
\end{array}\right),e_{2}=\left(\begin{array}{c}
0\\
1
\end{array}\right).
\]
As a special case of the PDEs derived by Lyons-Ni \cite{LyonsNi},
we have that 

\begin{align}
\Delta\left(\Phi_{\mathbb{D}}(z)\right) & =-\left(\sum_{i=1}^{d}e_{i}\otimes e_{i}\right)\otimes\Phi_{\mathbb{D}}(z)-2\sum_{i=1}^{d}e_{i}\otimes\frac{\partial\Phi_{\mathbb{D}}(z)}{\partial z_{i}}, & \forall z\in\mathbb{D}\label{eq:LyonsNiPDE}\\
\lim_{t\uparrow\tau_{\mathbb{D}}}\Phi_{\mathbb{D}}(B_{t}) & =\mathbf{1}\;\mathrm{a.s.}\mathbb{P}^{z} & \forall z\in\mathbb{D}\label{eq:Boundary}\\
\rho_{0}\big[\Phi_{\mathbb{D}}(z)\big] & =1,\rho_{1}\big[\Phi_{\mathbb{D}}(z)\big]=0 & \forall z\in\overline{\mathbb{D}}.\nonumber 
\end{align}
Recall that $\Phi_{\mathbb{D}}(z)$ is the expected signature of Brownian
motion starting point $z$ and up to the exit time of the planar unit
disc $\mathbb{D}\subseteq\mathbb{R}^{2}$. It has been established
in Theorem 3.5 in \cite{LyonsNi} that the $\rho_{n}\big(\Phi_{\mathbb{D}}(z)\big)$
is a polynomial in $z$ for each $n$ and hence continuous on $\overline{\mathbb{D}}$
and therefore the condition (\ref{eq:Boundary}) simplifies to 
\[
\Phi_{\mathbb{D}}(z)=\mathbf{1}\;\forall z\in\partial\mathbb{D},
\]
where $\partial\mathbb{D}$ is the boundary of the disc, $\partial\mathbb{D}=\{z=(x,y)^{T}:x^{2}+y^{2}=1\}$.
The goal of this work is to explore to what extent can the boundary
value problem (\ref{eq:LyonsNiPDE}) and (\ref{eq:Boundary}) be ``solved''.

Given a linear map $A:\mathbb{R}^{2}\rightarrow\mathbb{R}^{2}$, we
are interested in the extension of $A$ to a linear map $\mathbf{A}:T((\mathbb{R}^{2}))\rightarrow T((\mathbb{R}^{2}))$,
due to the following lemma:
\begin{lem}
\label{lem:SignatureLinearMap}Let $A:\mathbb{R}^{2}\rightarrow\mathbb{R}^{2}$
be a linear map. Let $\mathbf{A}:T((\mathbb{R}^{2}))\rightarrow T((\mathbb{R}^{2}))$
be the unique linear map such that $\mathbf{A}[\mathbf{1}]=\mathbf{1}$
and for all $u_{1},\ldots,u_{n}\in\mathbb{R}^{2}$,
\[
\mathbf{A}[u_{1}\otimes\cdots\otimes u_{n}]=A[u_{1}]\otimes\cdots\otimes A[u_{n}].
\]
Then 
\[
\mathbb{E}[S(A[B])_{0,\tau_{\mathbb{D}}}]=\mathbf{A}[\Phi_{\mathbb{D}}(u)].
\]
\end{lem}

\begin{proof}
Note that 
\begin{align*}
 & \rho_{n}\left[S(A[B])_{0,\tau_{\mathbb{D}}}\right]\\
= & \mathbb{E}\left[\int_{0<t_{1}<\ldots<t_{n}<\tau_{\mathbb{D}}}\circ\mathrm{d}A[B_{t_{1}}^{z}]\otimes\cdots\otimes\circ\mathrm{d}A[B_{t_{n}}^{z}]\right]\\
= & \mathbf{A}\mathbb{E}\left[\int_{0<t_{1}<\ldots<t_{n}<\tau_{\mathbb{D}}}\mathrm{d}B_{t_{1}}^{z}\otimes\cdots\otimes\mathrm{d}B_{t_{n}}^{z}\right]\quad\text{by linearity of integrals and expectation}\\
= & \rho_{n}\left[\mathbf{A}[\Phi_{\mathbb{D}}(z)]\right].
\end{align*}
\end{proof}

\section{Proof of the main result }

As a first step towards solving the PDE (\ref{eq:LyonsNiPDE}), we
will use the rotation invariance of Brownian motion to reduce (\ref{eq:LyonsNiPDE})
to an ordinary differential equation. We will use the following consequence
of rotational invariance, which is a reformulation of Lemma 3.3 in
\cite{LyonsNi}:
\begin{lem}
\label{lem:Rotational-Invariance}(Rotation property) For $\theta\in[0,2\pi),$let
$R(\theta):\mathbb{R}^{2}\rightarrow\mathbb{R}^{2}$ be the rotation
map 
\[
R(\theta):\left(\begin{array}{c}
x\\
y
\end{array}\right)\rightarrow\left(\begin{array}{cc}
\cos\theta & -\sin\theta\\
\sin\theta & \cos\theta
\end{array}\right)\left(\begin{array}{c}
x\\
y
\end{array}\right).
\]
Define the linear map $\mathbf{R}\left(\theta\right):T((\mathbb{R}^{2}))\rightarrow T((\mathbb{R}^{2}))$
such that $\mathbf{R}(\theta)[\mathbf{1}]=\mathbf{1}$ and for all
$u_{1},\ldots,u_{n}\in\mathbb{R}^{2}$,
\[
\mathbf{R}(\theta)[u_{1}\otimes\cdots\otimes u_{n}]=[R(\theta)u_{1}]\otimes\cdots\otimes[R(\theta)u_{n}].
\]
Then for all $z\in\mathbb{D}$, 
\[
\Phi_{\mathbb{D}}(R(\theta)z)=\mathbf{R}(\theta)\left[\Phi_{\mathbb{D}}(z)\right].
\]
\end{lem}

\begin{proof}
Let $B^{z}$ be two dimensional standard Brownian motion starting
at $z$. Then rotational invariance of Brownian motion states that,
as a stochastic process, we have the following equality in distribution:
\begin{equation}
B^{R(\theta)z}\overset{D}{=}R(\theta)B^{z}\label{eq:RotationalInvariance}
\end{equation}
 This means 
\begin{align*}
\mathbf{R}(\theta)\left[\Phi_{\mathbb{D}}(z)\right] & =\mathbb{E}\left[S(R(\theta)B^{z})_{0,\tau_{\mathbb{D}}}\right]\quad\text{by Lemma (\ref{lem:SignatureLinearMap})}\\
 & =\mathbb{E}\left[S(B^{R(\theta)z})_{0,\tau_{\mathbb{D}}}\right]\quad\text{by (\ref{eq:RotationalInvariance})}\\
 & =\Phi_{\mathbb{D}}(R(\theta)z).
\end{align*}
\end{proof}
The Rotation Property enables us to apply the separation of variables
method to solving Lyons-Ni PDE (\ref{eq:LyonsNiPDE}), through the
following representation:
\begin{cor}
(Separation of Variables) For all $r\in[0,1)$, $\theta\in\mathbb{R}$
and 
\[
z=R(\theta)\left(\begin{array}{c}
r\\
0
\end{array}\right),
\]
then
\begin{equation}
\Phi_{\mathbb{D}}(z)=\mathbf{R}(\theta)\left[\Phi_{\mathbb{D}}\left(\begin{array}{c}
r\\
0
\end{array}\right)\right].\label{eq:SeparationOfVariables}
\end{equation}
\end{cor}

\begin{proof}
By taking the ``$z$'' in Lemma \ref{lem:Rotational-Invariance}
as $(r\quad0)^{T}$. 
\end{proof}
We now substitute the separation of variables (\ref{eq:SeparationOfVariables})
into (\ref{eq:LyonsNiPDE}):
\begin{lem}
Define a linear map $f:T((\mathbb{R}^{2}))\rightarrow T((\mathbb{R}^{2}))$
by $f[\mathbf{1}]=0$ and 
\begin{align*}
f[u_{1}\otimes\cdots\otimes u_{n}]= & [R(\frac{\pi}{2})u_{1}]\otimes u_{2}\otimes\cdots\otimes u_{n}+u_{1}\otimes[R(\frac{\pi}{2})u_{2}]\otimes u_{3}\otimes\cdots\otimes u_{n}\\
 & +\cdots+u_{1}\otimes u_{2}\otimes\cdots\otimes[R(\frac{\pi}{2})u_{n}].
\end{align*}
Let $f^{2}:T((\mathbb{R}^{2}))\rightarrow T((\mathbb{R}^{2}))$ denote
the composition $f^{2}:=f\circ f$. Let 
\[
\phi(r)=\Phi_{\mathbb{D}}\left(\begin{array}{c}
r\\
0
\end{array}\right).
\]
`Then $\phi\left(r\right)$ is a solution to boundary value problem:
\begin{align}
 & r^{2}\phi^{\prime\prime}(r)+r\phi^{\prime}(r)+f^{2}[\phi(r)]\nonumber \\
= & -r^{2}\left(e_{1}\otimes e_{1}+e_{2}\otimes e_{2}\right)\otimes\phi\left(r\right)-2r\left(re_{1}\otimes\phi^{\prime}\left(r\right)+e_{2}\otimes f[\phi\left(r\right)]\right)\label{eq:ODE1}
\end{align}
\begin{equation}
\phi(1)=\mathbf{1}.\label{eq:BoundaryODE}
\end{equation}

Conversely, for any solution $\phi$ to the boundary problem (\ref{eq:ODE1}),
the function $\mathbf{R}(\theta)[\phi(r)]$ solves Lyons-Ni' s PDE
(\ref{eq:LyonsNiPDE}) and (\ref{eq:Boundary}) in polar coordinates.
\end{lem}

\begin{proof}
Using the polar coordinate representation of Laplacian and partial
derivatives along standard basis: 
\begin{align*}
\Delta= & \frac{\partial^{2}}{\partial r^{2}}+\frac{1}{r}\frac{\partial}{\partial r}+\frac{1}{r^{2}}\frac{\partial^{2}}{\partial\theta^{2}}\\
\frac{\partial}{\partial z_{1}}= & \cos\theta\frac{\partial}{\partial r}-\frac{\sin\theta}{r}\frac{\partial}{\partial\theta}\\
\frac{\partial}{\partial z_{2}}= & \sin\theta\frac{\partial}{\partial r}+\frac{\cos\theta}{r}\frac{\partial}{\partial\theta},
\end{align*}
we have that if 
\[
u(r,\theta)=\mathbf{\Phi_{\mathbb{D}}}(R(\theta)\left(\begin{array}{c}
r\\
0
\end{array}\right)),
\]
then
\begin{align}
 & \left[\frac{\partial^{2}u}{\partial r^{2}}+\frac{1}{r}\frac{\partial u}{\partial r}+\frac{1}{r^{2}}\frac{\partial^{2}u}{\partial\theta^{2}}\right](r,\theta)\nonumber \\
= & -\left(\sum_{i=1}^{2}e_{i}\otimes e_{i}\right)\otimes u(r,\theta)-2(\cos\theta e_{1}+\sin\theta e_{2})\otimes\frac{\partial u}{\partial r}(r,\theta)\nonumber \\
 & -2(\frac{-\sin\theta e_{1}+\cos\theta e_{2}}{r})\otimes\frac{\partial u}{\partial\theta}.\label{eq:PolarLyonsNi}
\end{align}

By substituting the Separation of Variables
\[
u(r,\theta)=\mathbf{R}(\theta)[\phi(r)],
\]
which is equivalent to (\ref{eq:SeparationOfVariables}), into (\ref{eq:PolarLyonsNi}),
we see that 
\begin{align}
 & \mathbf{R}(\theta)\left[\phi^{\prime\prime}+\frac{1}{r}\phi^{\prime}\right](r)+\frac{1}{r^{2}}\frac{\partial^{2}}{\partial\theta^{2}}\mathbf{R}(\theta)[\phi](r)\nonumber \\
= & -\left(\sum_{i=1}^{2}e_{i}\otimes e_{i}\right)\otimes\mathbf{R}(\theta)[\phi(r)]-2(\cos\theta e_{1}+\sin\theta e_{2})\otimes\mathbf{R}(\theta)[\phi^{\prime}(r)]\nonumber \\
 & -2(\frac{-\sin\theta e_{1}+\cos\theta e_{2}}{r})\otimes\frac{\partial}{\partial\theta}\mathbf{R}(\theta)[\phi(r)].\label{eq:}
\end{align}
Note that for any $u_{1},\ldots,u_{n}\in\mathbb{R}^{2}$,
\begin{align}
\frac{\partial}{\partial\theta}\mathbf{R}(\theta)[v_{1}\otimes\cdots\otimes v_{n}]= & \frac{\partial}{\partial\theta}\left\{ [R(\theta)u_{1}]\otimes\cdots\otimes[R(\theta)u_{n}]\right\} \nonumber \\
= & [R^{\prime}(\theta)u_{1}]\otimes[R(\theta)u_{2}]\otimes\cdots\otimes[R(\theta)u_{n}]+\cdots\nonumber \\
 & [R(\theta)u_{1}]\otimes\cdots\otimes[R(\theta)u_{n-1}]\otimes R^{\prime}(\theta)u_{n},\label{eq:DiffRot}
\end{align}
where 
\[
R^{\prime}(\theta)=\left(\begin{array}{cc}
-\sin\theta & -\cos\theta\\
\cos\theta & -\sin\theta
\end{array}\right)=R(\theta)\left(\begin{array}{cc}
0 & -1\\
1 & 0
\end{array}\right)=R(\theta)R(\frac{\pi}{2}).
\]
Therefore (\ref{eq:DiffRot}) becomes
\begin{align*}
\frac{\partial}{\partial\theta}\mathbf{R}(\theta)[u_{1}\otimes\cdots\otimes u_{n}]= & [R(\theta)R(\frac{\pi}{2})u_{1}]\otimes[R(\theta)u_{2}]\otimes\cdots\otimes[R(\theta)u_{n}]+\cdots\\
 & +[R(\theta)u_{1}]\otimes\cdots\otimes[R(\theta)u_{n-1}]\otimes R(\theta)R(\frac{\pi}{2})u_{n}\\
= & \mathbf{R}(\theta)f[u_{1}\otimes\cdots\otimes u_{n}].
\end{align*}
We have 
\begin{align}
 & \mathbf{R}(\theta)\left[\phi^{\prime\prime}+\frac{1}{r}\phi^{\prime}\right](r)+\frac{1}{r^{2}}\mathbf{R}(\theta)f^{2}[\phi(r)]\nonumber \\
= & -\left(\sum_{i=1}^{2}e_{i}\otimes e_{i}\right)\otimes\mathbf{R}(\theta)[\phi(r)]-2(\cos\theta e_{1}+\sin\theta e_{2})\otimes\mathbf{R}(\theta)[\phi^{\prime}(r)]\nonumber \\
 & -2(\frac{-\sin\theta e_{1}+\cos\theta e_{2}}{r})\otimes\mathbf{R}(\theta)f[\phi(r)].\label{eq:ODEPreRotation}
\end{align}
Noting that 
\begin{align}
\cos\theta e_{1}+\sin\theta e_{2} & =\mathbf{R}(\theta)e_{1}\nonumber \\
-\sin\theta e_{1}+\cos\theta e_{2} & =\mathbf{R}(\theta)e_{2}\nonumber \\
\sum_{i=1}^{2}e_{i}\otimes e_{i} & =\mathbf{R}(\theta)\left(\sum_{i=1}^{2}e_{i}\otimes e_{i}\right),\label{eq:RotationRep}
\end{align}
we have 
\begin{align}
 & \mathbf{R}(\theta)\left[\phi^{\prime\prime}+\frac{1}{r}\phi^{\prime}\right](r)+\frac{1}{r^{2}}\mathbf{R}(\theta)f^{2}[\phi(r)]\nonumber \\
= & -\mathbf{R}(\theta)\left(\sum_{i=1}^{2}e_{i}\otimes e_{i}\right)\otimes\mathbf{R}(\theta)[\phi(r)]-2\mathbf{R}(\theta)e_{1}\otimes\mathbf{R}(\theta)[\phi^{\prime}(r)]-2(\frac{\mathbf{R}(\theta)e_{2}}{r})\otimes\mathbf{R}(\theta)f[\phi(r)].\label{eq:RotatedEquation}
\end{align}
Note that $\mathbf{R}(\theta)$ is invertible with inverse map $\mathbf{R}(-\theta)$.
We now apply $\mathbf{R}(-\theta)$ on both sides of (\ref{eq:RotatedEquation})
and use the fact that for any tensor series $a$ and $b$ in $T((\mathbb{R}^{2}))$,
$\mathbf{R}(-\theta)[a\otimes b]=\mathbf{R}(-\theta)[a]\otimes\mathbf{R}(-\theta)[b]$
to obtain (\ref{eq:ODE1}).

Since $\Phi_{\mathbb{D}}(z)=\mathbf{1}$ for all $z\in\partial\mathbb{D}$,
we have $\phi(1)=\Phi_{\mathbb{D}}\left(\begin{array}{c}
1\\
0
\end{array}\right)=\mathbf{1}$.

Conversely, assume $\phi$ is any solution of (\ref{eq:ODE1}) and
(\ref{eq:BoundaryODE}). Note that the derivation between (\ref{eq:ODE1})
and (\ref{eq:RotatedEquation}) meant that 
\[
u(r,\theta)=\mathbf{\mathbf{R}}(\theta)[\phi(r)]
\]
satisfies (\ref{eq:BoundaryODE}) if and only if $\phi$ satisfies
(\ref{eq:ODE1}). If $z\in\partial\mathbb{D}$, then
\[
\Phi_{\mathbb{D}}\left(z\right)=\mathbf{R}(\theta)\phi(1)=\mathbf{R}(\theta)[\mathbf{1}]=\mathbf{1}.
\]
\end{proof}
We now try to solve the ODE (\ref{eq:ODE1}) by looking for a power
series solution (also known as the Frobenius method). 
\begin{lem}
\label{lem:TensorCoefficient}There exists a sequence $(a_{n})_{n=0}^{\infty}$,
where $a_{n}\in T((\mathbb{R}^{2}))$, such that 
\begin{equation}
\phi(r)=\sum_{n=0}^{\infty}a_{n}r^{n},\label{eq:SeriesExpansion-1}
\end{equation}
with $(a_{n})_{n=0}^{\infty}$ satisfying 
\begin{equation}
\rho_{i}[a_{n}]=0\qquad\forall n>i.\label{eq:iLessThann}
\end{equation}
\begin{equation}
f^{2}[a_{0}]=0.\label{eq:r0Coeff}
\end{equation}

\begin{equation}
a_{1}+f^{2}[a_{1}]=-2e_{2}\otimes f[a_{0}].\label{eq:r1Coeff}
\end{equation}
For $n\geq2$,

\begin{align}
 & n^{2}a_{n}+f^{2}[a_{n}]\nonumber \\
= & -\left(e_{1}\otimes e_{1}+e_{2}\otimes e_{2}\right)\otimes a_{n-2}-2\left((n-1)e_{1}\otimes a_{n-1}+e_{2}\otimes f[a_{n-1}]\right).\label{eq:rnCoeff}
\end{align}
The series (\ref{eq:SeriesExpansion-1}) has no convergence issue
because (\ref{eq:iLessThann}) implies that when restricted on each
component $(\mathbb{R}^{2})^{\otimes i}$, the sum (\ref{eq:SeriesExpansion-1})
is a finite sum .
\end{lem}

\begin{rem}
For each $n$, $a_{n}$ is a tensor \textbf{series}. Note in particular
that $a_{n}$ is generally contains terms in $(\mathbb{R}^{2})^{\otimes k}$for
all $k$, not just $(\mathbb{R}^{2})^{\otimes n}$.
\end{rem}

\begin{proof}
Theorem 3.5 in \cite{LyonsNi} states that for each $i$, $\rho_{i}[\Phi_{\mathbb{D}}\left(z_{1},z_{2}\right)]$
is a polynomial in $z_{1}$ and $z_{2}$ of degree $n$ when $(z_{1},z_{2})^{T}\in\mathbb{D}$.
In particular, $\rho_{i}[\phi(r)]=\rho_{i}[\Phi_{\mathbb{D}}(r,0)]$
is a polynomial in $r$ of degree at most $n$. If we define $a_{n}$
such that $\rho_{i}[a_{n}]$ is the coefficient of $r^{n}$ in $\rho_{i}[\phi(r)]$,
then $\rho_{i}[a_{n}]=0$ for all $n>i$ and 
\begin{equation}
\phi(r)=\sum_{n=0}^{\infty}a_{n}r^{n}.\label{eq:SeriesExpansion}
\end{equation}

If we substitute the series expansion (\ref{eq:SeriesExpansion})
into (\ref{eq:ODE1}) and equate the coefficient of $r^{n}$ for the
case $n=0$, $n=1$ and $n\geq2$, we obtain (\ref{eq:r0Coeff}),
(\ref{eq:r1Coeff}) and (\ref{eq:rnCoeff}) respectively.
\end{proof}
We must now understand how the linear map $f$ acts on $a_{n}$. We
will do so by trying to expand $a_{n}$ in terms of eigenvectors for
$f$. 

Note that $f$ is defined in terms of $R(\frac{\pi}{2})$ and so we
start by finding the eigenvectors for $R(\frac{\pi}{2})$. If $\mathrm{i}^{2}=-1$,
then the linear map 
\[
R(\frac{\pi}{2}):\left(\begin{array}{c}
x\\
y
\end{array}\right)\rightarrow\left(\begin{array}{cc}
0 & -1\\
1 & 0
\end{array}\right)\left(\begin{array}{c}
x\\
y
\end{array}\right)
\]
 has eigenvectors 
\begin{equation}
v_{1}=\left(\begin{array}{c}
\mathrm{i}\\
1
\end{array}\right),v_{2}=\left(\begin{array}{c}
-\mathrm{i}\\
1
\end{array}\right)\label{eq:v1v2}
\end{equation}
with eigenvalues $\mathrm{i}$ and $-\mathrm{i}$ respectively. 

This means if $i_{1},\ldots,i_{j}\in\{1,2\}$, 
\begin{align*}
f[v_{i_{1}}\otimes\cdots\otimes v_{i_{m}}]= & R(\frac{\pi}{2})v_{i_{1}}\otimes v_{i_{2}}\otimes\cdots\otimes v_{i_{j}}+v_{i_{1}}\otimes R(\frac{\pi}{2})v_{i_{2}}\otimes v_{i_{3}}\otimes\cdots\otimes v_{i_{j}}\\
 & +\cdots+v_{i_{1}}\otimes\cdots\otimes v_{i_{j-1}}\otimes R(\frac{\pi}{2})v_{i_{j}}\\
= & \sum_{m=1}^{n}(-1)^{i_{m}+1}\mathrm{i}v_{i_{1}}\otimes\cdots\otimes v_{i_{j}}.
\end{align*}
Therefore for all $i_{1},\ldots,i_{j}\in\{1,2\}$, $v_{i_{1}}\otimes\cdots\otimes v_{i_{j}}$is
an eigenvector for $f$ with eigenvalue $\lambda_{i_{1},\ldots,i_{j}}=$$\left(|\{m:i_{m}=1\}|-|\{m:i_{m}=2\}|\right)\mathrm{i}$.
Note in particular that 
\[
|\lambda_{i_{1},\ldots,i_{j}}|\leq j.
\]

Note that every element $\mathbf{v}\in(\mathbb{R}^{2})^{\otimes n}$
may be expanded in terms of 
\[
\{v_{i_{1}}\otimes\cdots\otimes v_{i_{j}}:i_{m}\in\{1,2\}\}
\]
with coefficients in $\mathbb{C}$. This is because 
\[
e_{1}=-\frac{\mathrm{i}}{2}(v_{1}+v_{2}),\;e_{2}=\frac{1}{2}(v_{1}+v_{2}).
\]
and that any element of $(\mathbb{R}^{2})^{\otimes j}$ can be expanded
in terms of $e_{1}$ and $e_{2}$. Let 
\[
V_{j}^{\beta}=\mathrm{span}_{\mathbb{C}}\{v_{i_{1}}\otimes\cdots\otimes v_{i_{j}}:i_{m}\in\{1,2\},\lambda_{i_{1},\cdots,i_{j}}=\beta\mathrm{i}\}.
\]
We use the convention that 
\[
V_{0}^{\beta}=\begin{cases}
\{0\}, & \text{if }\beta\neq0,\\
\mathbb{R}, & \text{if }\beta=0.
\end{cases}
\]
The reason for the choice of this convention is that if $v\in V_{n}^{\alpha}$
and $w\in V_{m}^{\beta}$, then 
\[
v\otimes w\in V_{n+m}^{\alpha+\beta}.
\]
We define 
\[
V^{\beta}=\left\{ a\in T((\mathbb{R}^{2})):\rho_{i}(a)\in V_{i}^{\beta}\;\forall i\right\} .
\]
Let $\rho^{\beta}$ be the projection of $T((\mathbb{R}^{2}))$ onto
$V^{\beta}$ and $\rho_{i}^{\beta}$ be the projection of $T((\mathbb{R}^{2}))$
on to $V_{i}^{\beta}$. We have the decomposition
\begin{equation}
a_{n}=\sum_{\beta\in\mathbb{Z}}a_{n}^{\beta},\label{eq:Decomposition}
\end{equation}
with $a_{n}^{\beta}\in V^{\beta}$. Equation (\ref{eq:Decomposition})
should be understood as 
\[
\rho_{N}[a_{n}]=\sum_{\beta\in\mathbb{Z}}\rho_{N}[a_{n}^{\beta}],
\]
where the sum is finite since, by the definition of the eigenvalue
$\beta$, $\rho_{N}[a_{n}^{\beta}]=0$ when $|\beta|>N$.

We now write down a recurrence for $a_{n}^{\beta}$. 
\begin{lem}
\label{lem:EigenRecurrence}Let $(a_{n})_{n=0}^{\infty}$ be any sequence
of tensor series described in Lemma \ref{lem:TensorCoefficient}.
Suppose that 
\begin{equation}
a_{n}=\sum_{\beta\in\mathbb{Z}}a_{n}^{\beta}.\label{eq:EigenExpansion}
\end{equation}
Then 
\begin{equation}
a_{0}^{\beta}=0\qquad\forall\beta\neq0\label{eq:a0beta}
\end{equation}
\begin{equation}
a_{1}^{\beta}=0\qquad\forall\beta^{2}\neq1\label{eq:a1beta}
\end{equation}
For $n\geq2$, 

\begin{align}
 & (n^{2}-\beta^{2})a_{n}^{\beta}\label{eq:anbeta}\\
= & -\frac{1}{2}\left(v_{1}\otimes v_{2}+v_{2}\otimes v_{1}\right)\otimes a_{n-2}^{\beta}+(n-\beta)\mathrm{i}v_{1}\otimes a_{n-1}^{\beta-1}-(n+\beta)\mathrm{i}v_{2}\otimes a_{n-1}^{\beta+1}.
\end{align}
\end{lem}

\begin{proof}
Note that with $v_{1}$ and $v_{2}$ as defined in (\ref{eq:v1v2}),
then 
\begin{equation}
e_{1}=-\mathrm{i}\frac{v_{1}-v_{2}}{2},\quad e_{2}=\frac{v_{1}+v_{2}}{2}.\label{eq:e1e2v1v2}
\end{equation}

Substituting (\ref{eq:e1e2v1v2}) and $a_{n}=\sum_{\beta}a_{n}^{\beta}$
into (\ref{eq:r0Coeff}), (\ref{eq:r1Coeff}) and (\ref{eq:rnCoeff})
we have 

\begin{equation}
\sum_{\beta}(-\beta^{2})a_{0}^{\beta}=0.\label{eq:r0Coeff-1}
\end{equation}

\begin{equation}
\sum_{\beta}(1-\beta^{2})a_{1}^{\beta}=-2(v_{1}+v_{2})\otimes\sum_{\beta}\beta\mathrm{i}a_{0}^{\beta}.\label{eq:r1Coeff-1}
\end{equation}
By projecting onto $V^{\beta}$, we obtain (\ref{eq:a0beta}) and
(\ref{eq:a1beta}). 

For $n\geq2$,
\begin{align}
 & \sum_{\beta}(n^{2}-\beta^{2})a_{n}^{\beta}\nonumber \\
= & -\frac{1}{2}\left(v_{1}\otimes v_{2}+v_{2}\otimes v_{1}\right)\otimes\sum_{\beta}a_{n-2}^{\beta}\label{eq:rnCoeff-2}\\
 & -\left(-\mathrm{i}(v_{1}-v_{2})\otimes(n-1)\sum_{\beta}a_{n-1}^{\beta}+(v_{1}+v_{2})\otimes\sum_{\beta}\beta\mathrm{i}a_{n-1}^{\beta}\right)\\
= & -\frac{1}{2}\left(v_{1}\otimes v_{2}+v_{2}\otimes v_{1}\right)\otimes\sum_{\beta}a_{n-2}^{\beta}+\mathrm{i}v_{1}\otimes\big[(n-1-(\beta-1))\sum_{\beta}a_{n-1}^{\beta-1}\big]\\
 & +\mathrm{i}v_{2}\otimes\big[(n-1+(\beta+1))\sum_{\beta}a_{n-1}^{\beta+1}\big].
\end{align}

Projecting to $V^{\beta}$ on both sides, we obtain (\ref{eq:anbeta}). 
\end{proof}
We note from (\ref{eq:a0beta}) and (\ref{eq:a1beta}) that most tensor
terms of $a_{0}$ and $a_{1}$ are zeros. There are natural generalisation
of this observation for generic $n$. 
\begin{cor}
\label{cor:UpperLowerLimit}If $n<|\beta|$, then $a_{n}^{\beta}=0$. 
\end{cor}

\begin{proof}
We will prove by induction on $n$ using the recurrence 
\begin{align*}
 & (n^{2}-\beta^{2})a_{n}^{\beta}\\
= & -\frac{1}{2}\left(v_{1}\otimes v_{2}+v_{2}\otimes v_{1}\right)\otimes a_{n-2}^{\beta}+(n-\beta)\mathrm{i}v_{1}\otimes a_{n-1}^{\beta-1}-(n+\beta)\mathrm{i}v_{2}\otimes a_{n-1}^{\beta+1}.
\end{align*}

The base induction cases $n=0$ and $n=1$ are part of Lemma \ref{lem:EigenRecurrence}.

If $n<|\beta|$, then $n-2<|\beta|$ and $n-1<|\beta-1|$ and $n-1<|\beta+1|$,
so by induction hypothesis, $a_{n-1}^{\beta-1}=a_{n-2}^{\beta}=a_{n-1}^{\beta+1}=0$. 
\end{proof}
As we now see, the symmetry of Brownian motion means that unless $n$
and $\beta$ have the same parity, $a_{n}^{\beta}$ is in fact equal
to zero. 
\begin{lem}
\label{lem:SameParity}If $n-\beta$ is odd, then $a_{n}^{\beta}=0$. 
\end{lem}

\begin{proof}
Applying Lemma \ref{lem:SignatureLinearMap} with the linear map $A(z)=-z$,
we have 

\begin{align*}
\mathbf{A}[\phi(r)] & =\mathbf{A}[\Phi_{\mathbb{D}}\left(\begin{array}{c}
r\\
0
\end{array}\right)]\\
 & =\mathbb{E}[S(A[B^{(r,0)}])_{0,\tau_{\mathbb{D}}}]\\
 & =\mathbb{E}[S(-B^{(r,0)})_{0,\tau_{\mathbb{D}}}]\\
 & =\mathbb{E}[S(B^{(-r,0)})_{0,\tau_{\mathbb{D}}}]\\
 & =\phi(-r).
\end{align*}
Note that for any $a\in V^{\beta}$, $\mathbf{A}[a]=(-1)^{\beta}a$.
Since $\mathbf{A}[a_{n}^{\beta}]=(-1)^{\beta}a_{n}^{\beta}$, we may
substitute $\phi(r)=\sum_{n=0}^{\infty}\sum_{\beta}a_{n}^{\beta}r^{n}$
to get 
\[
\sum_{n=0}^{\infty}\sum_{\beta}(-1)^{\beta}a_{n}^{\beta}r^{n}=\sum_{n=0}^{\infty}\sum_{\beta}a_{n}^{\beta}(-r)^{n}.
\]
By equating the coefficient of $r^{n}$ and projecting to $V^{\beta}$,
$(-1)^{\beta}a_{n}^{\beta}=(-1)^{n}a_{n}^{\beta}$, or equivalently
that $a_{n}^{\beta}=0$ if $n-\beta$ is odd. 
\end{proof}
We now start to solve (\ref{eq:anbeta}). 
\begin{lem}
Suppose that $(a_{n}^{\beta})_{n\in\mathbb{N}\cup\{0\},\beta\in\mathbb{Z}}$
satisfies the recurrence relations (\ref{eq:a0beta}), (\ref{eq:a1beta})
and (\ref{eq:anbeta}). For $n\geq|\beta|$, if we define $(b_{n}^{\beta})_{n\in\mathbb{N}\cup\{0\},\beta\in\mathbb{Z}}$
by
\begin{equation}
a_{n}^{\beta}=\frac{(-1)^{\frac{n-\beta}{2}}}{(\frac{n-\beta}{2})!(\frac{n+\beta}{2})!}\frac{1}{2^{n}}b_{n}^{\beta}\label{eq:Defineb}
\end{equation}
then for all $w\in\mathbb{C}$, 
\begin{align}
 & \sum_{\beta}\sum_{n\geq|\beta|}b_{n}^{\beta}w^{n}\nonumber \\
= & (\mathbf{1}-\mathrm{i}w(v_{1}+v_{2})-\frac{w^{2}}{2}\left(v_{1}\otimes v_{2}+v_{2}\otimes v_{1}\right))^{-1}\otimes\label{eq:BnGenerating}\\
 & \left[(\mathbf{1}-\mathrm{i}wv_{1})\otimes\sum_{\beta=1}^{\infty}b_{|\beta|}^{\beta}w^{\beta}+(\mathbf{1}-\mathrm{i}wv_{2})\otimes\sum_{\beta=-1}^{-\infty}b_{|\beta|}^{\beta}w^{|\beta|}+(\mathbf{1}-\mathrm{i}wv_{1}-\mathrm{i}wv_{2})\otimes b_{0}^{0}\right]
\end{align}
\end{lem}

\begin{rem}
The infinite sum $\sum_{\beta}\sum_{n\geq|\beta|}b_{n}^{\beta}w^{n}$
is well-defined because $\sum_{\beta}\sum_{n\geq|\beta|}\rho_{i}[b_{n}^{\beta}]w^{n}$
is a finite sum
\[
\sum_{\beta}\sum_{n\geq|\beta|}\rho_{i}[b_{n}^{\beta}]w^{n}=\sum_{\beta=-i}^{i}\sum_{n=|\beta|}^{i}\rho_{i}[b_{n}^{\beta}]w^{n},
\]
due to Corollary \ref{cor:UpperLowerLimit} and (\ref{eq:iLessThann}).
Likewise, 
\[
\sum_{\beta=1}^{\infty}\rho_{i}[b_{|\beta|}^{\beta}]w^{\beta}=\sum_{\beta=1}^{i}\rho_{i}[b_{|\beta|}^{\beta}]w^{\beta}\quad\sum_{\beta=-1}^{-\infty}\rho_{i}[b_{|\beta|}^{\beta}]w^{|\beta|}=\sum_{\beta=-1}^{-i}\rho_{i}[b_{|\beta|}^{\beta}]w^{|\beta|}.
\]
\end{rem}

\begin{rem}
We will determine the unknown coefficients $(b_{|k|}^{k})_{k\in\mathbb{Z}}$
using the boundary condition later on. 
\end{rem}

\begin{proof}
Substituting (\ref{eq:Defineb}) into (\ref{eq:anbeta}) to get for
$n\geq|\beta|+2$,
\begin{align*}
\frac{(-1)^{\frac{n-\beta}{2}}(n^{2}-\beta^{2})}{(\frac{n-\beta}{2})!(\frac{n+\beta}{2})!}\frac{1}{2^{n}}b_{n}^{\beta} & =-\frac{1}{2}\left(v_{1}\otimes v_{2}+v_{2}\otimes v_{1}\right)\otimes\frac{(-1)^{\frac{n-\beta}{2}-1}}{(\frac{n-\beta}{2}-1)!(\frac{n+\beta}{2}-1)!2^{n-2}}b_{n-2}^{\beta}\\
 & +(n-\beta)\mathrm{i}v_{1}\otimes\frac{(-1)^{\frac{n-\beta}{2}}}{(\frac{n-\beta}{2})!(\frac{n+\beta}{2}-1)!2^{n-1}}b_{n-1}^{\beta-1}\\
 & -(n+\beta)\mathrm{i}v_{2}\otimes\frac{(-1)^{\frac{n-\beta}{2}-1}}{(\frac{n-\beta}{2}-1)!(\frac{n+\beta}{2})!2^{n-1}}b_{n-1}^{\beta+1}.
\end{align*}
Simplifying, we have for $n\geq|\beta|+2$,
\begin{align*}
b_{n}^{\beta} & =\frac{1}{2}\left(v_{1}\otimes v_{2}+v_{2}\otimes v_{1}\right)\otimes b_{n-2}^{\beta}+\mathrm{i}v_{1}\otimes b_{n-1}^{\beta-1}+\mathrm{i}v_{2}\otimes b_{n-1}^{\beta+1}.
\end{align*}
Let $w\in\mathbb{C}$. Multiplying by $w^{n}$ and summing over all
$n\geq|\beta|+2$, we have 
\begin{align*}
\sum_{n\geq|\beta|+2}b_{n}^{\beta}w^{n} & =\frac{w^{2}}{2}\left(v_{1}\otimes v_{2}+v_{2}\otimes v_{1}\right)\otimes\sum_{n\geq|\beta|+2}b_{n-2}^{\beta}w^{n-2}+\mathrm{i}wv_{1}\otimes\sum_{n\geq|\beta|+2}b_{n-1}^{\beta-1}w^{n-1}\\
 & +\mathrm{i}wv_{2}\otimes\sum_{n\geq|\beta|+2}b_{n-1}^{\beta+1}w^{n-1},
\end{align*}
and therefore using that $b_{n}^{\beta}=0$ when $n-\beta$ is odd
(see Lemma \ref{lem:SameParity}),
\begin{align*}
\sum_{n\geq|\beta|}b_{n}^{\beta}w^{n} & =b_{|\beta|}^{\beta}w^{|\beta|}+\frac{w^{2}}{2}\left(v_{1}\otimes v_{2}+v_{2}\otimes v_{1}\right)\otimes\sum_{n\geq|\beta|}b_{n}^{\beta}w^{n}+\mathrm{i}wv_{1}\otimes\sum_{n\geq|\beta|+1}b_{n}^{\beta-1}w^{n}\\
 & +\mathrm{i}wv_{2}\otimes\sum_{n\geq|\beta|+1}b_{n}^{\beta+1}w^{n}.
\end{align*}
If $\beta\geq1$, then $|\beta|+1=\beta+1$ and $\beta-1=|\beta-1|$
. Hence for $\beta\geq1,$
\begin{align*}
\sum_{n\geq|\beta|}b_{n}^{\beta}w^{n} & =b_{|\beta|}^{\beta}w^{\beta}-\mathrm{i}wv_{1}\otimes b_{|\beta-1|}^{\beta-1}w^{|\beta-1|}+\frac{w^{2}}{2}\left(v_{1}\otimes v_{2}+v_{2}\otimes v_{1}\right)\otimes\sum_{n\geq|\beta|}b_{n}^{\beta}w^{n}\\
 & +\mathrm{i}wv_{1}\otimes\sum_{n\geq|\beta-1|}b_{n}^{\beta-1}w^{n}+\mathrm{i}wv_{2}\otimes\sum_{n\geq|\beta+1|}b_{n}^{\beta+1}w^{n}.
\end{align*}
For $\beta\leq-1$, then $|\beta|+1=|\beta-1|$, $|\beta|+1=1-\beta$
and $-1-\beta\geq0$. Hence for $\beta\leq-1$, 
\begin{align*}
\sum_{n\geq|\beta|}b_{n}^{\beta}w^{n} & =b_{|\beta|}^{\beta}w^{|\beta|}-\mathrm{i}wv_{2}\otimes b_{-\beta-1}^{\beta+1}w^{-\beta-1}+\frac{w^{2}}{2}\left(v_{1}\otimes v_{2}+v_{2}\otimes v_{1}\right)\otimes\sum_{n\geq|\beta|}b_{n}^{\beta}w^{n}\\
 & +\mathrm{i}wv_{1}\otimes\sum_{n\geq|\beta-1|}b_{n}^{\beta-1}w^{n}+\mathrm{i}wv_{2}\otimes\sum_{n\geq|\beta+1|}b_{n}^{\beta+1}w^{n}.
\end{align*}
For $\beta=0$: 
\begin{align*}
\sum_{n\geq|\beta|}b_{n}^{\beta}w^{n} & =b_{|\beta|}^{\beta}w^{|\beta|}+\frac{w^{2}}{2}\left(v_{1}\otimes v_{2}+v_{2}\otimes v_{1}\right)\otimes\sum_{n\geq|\beta|}b_{n}^{\beta}w^{n}\\
 & +\mathrm{i}wv_{1}\otimes\sum_{n\geq|\beta-1|}b_{n}^{\beta-1}w^{n}+\mathrm{i}wv_{2}\otimes\sum_{n\geq|\beta+1|}b_{n}^{\beta+1}w^{n}.
\end{align*}
Summing over all $\beta\in\mathbb{Z}$,
\begin{align*}
\sum_{\beta}\sum_{n\geq|\beta|}b_{n}^{\beta}w^{n} & =\sum_{\beta}b_{|\beta|}^{\beta}w^{\beta}-\mathrm{i}wv_{1}\otimes\sum_{\beta\geq1}b_{|\beta-1|}^{\beta-1}w^{\beta-1}-\mathrm{i}wv_{2}\otimes\sum_{\beta\leq-1}b_{-\beta-1}^{\beta+1}w^{-\beta-1}\\
 & +\frac{w^{2}}{2}\left(v_{1}\otimes v_{2}+v_{2}\otimes v_{1}\right)\otimes\sum_{\beta}\sum_{n\geq|\beta|}b_{n}^{\beta}w^{n}\\
 & +\mathrm{i}wv_{1}\otimes\sum_{\beta}\sum_{n\geq|\beta-1|}b_{n}^{\beta-1}w^{n}+\mathrm{i}wv_{2}\otimes\sum_{\beta}\sum_{n\geq|\beta+1|}b_{n}^{\beta+1}w^{n}.
\end{align*}
Noting that 
\[
\sum_{n\geq|\beta-1|}b_{n}^{\beta-1}w^{n}=\sum_{n\geq|\beta||}b_{n}^{\beta}w^{n},\quad\sum_{n\geq|\beta+1|}b_{n}^{\beta+1}w^{n}=\sum_{n\geq|\beta|}b_{n}^{\beta}w^{n},
\]
we have 
\begin{align*}
 & (\mathbf{1}-\mathrm{i}w(v_{1}+v_{2})-\frac{w^{2}}{2}\left(v_{1}\otimes v_{2}+v_{2}\otimes v_{1}\right))\otimes\sum_{\beta}\sum_{n\geq|\beta|}b_{n}^{\beta}w^{n}\\
= & \sum_{\beta}b_{|\beta|}^{\beta}w^{\beta}-\mathrm{i}wv_{1}\otimes\sum_{\beta\geq1}b_{|\beta-1|}^{\beta-1}w^{\beta-1}-\mathrm{i}wv_{2}\otimes\sum_{\beta\leq-1}b_{-\beta-1}^{\beta+1}w^{-\beta-1}.
\end{align*}
As the tenor series $(\mathbf{1}-\mathrm{i}w(v_{1}+v_{2})-\frac{w^{2}}{2}\left(v_{1}\otimes v_{2}+v_{2}\otimes v_{1}\right))$
is invertible with respect $\otimes$, 
\begin{align*}
 & \sum_{\beta}\sum_{n\geq|\beta|}b_{n}^{\beta}w^{n}\\
= & (\mathbf{1}-\mathrm{i}w(v_{1}+v_{2})-\frac{w^{2}}{2}\left(v_{1}\otimes v_{2}+v_{2}\otimes v_{1}\right))^{-1}\\
 & \otimes\left[\sum_{\beta}b_{|\beta|}^{\beta}w^{\beta}-\mathrm{i}wv_{1}\otimes\sum_{\beta\geq0}b_{|\beta|}^{\beta}w^{\beta}-\mathrm{i}wv_{2}\otimes\sum_{\beta\leq0}b_{|\beta|}^{\beta}w^{-\beta}\right]\\
= & (\mathbf{1}-\mathrm{i}w(v_{1}+v_{2})-\frac{w^{2}}{2}\left(v_{1}\otimes v_{2}+v_{2}\otimes v_{1}\right))^{-1}\\
 & \otimes\left[(\mathbf{1}-\mathrm{i}wv_{1})\otimes\sum_{\beta=1}^{\infty}b_{|\beta|}^{\beta}w^{\beta}+(\mathbf{1}-\mathrm{i}wv_{2})\otimes\sum_{\beta=-1}^{-\infty}b_{|\beta|}^{\beta}w^{|\beta|}+(\mathbf{1}-\mathrm{i}wv_{1}-\mathrm{i}wv_{2})\otimes b_{0}^{0}\right].
\end{align*}
\end{proof}
\begin{lem}
For all $\beta\in\mathbb{Z}$ and $n\in\mathbb{N}\cup\{0\}$,
\begin{align}
b_{n}^{\beta} & =b_{|\beta|}^{\beta}1_{n=|\beta|}+\sum_{k=-n}^{n}\big\{\rho_{n-|k|}^{\beta-k}\big((\mathbf{1}-\mathrm{i}(v_{1}+v_{2})-\frac{1}{2}\left(v_{1}\otimes v_{2}+v_{2}\otimes v_{1}\right))^{-1}\label{eq:bnFormula}\\
 & \otimes[\frac{1}{2}(v_{1}\otimes v_{2}+v_{2}\otimes v_{1})+\mathrm{i}v_{2}1_{\{k\geq1\}}+\mathrm{i}v_{1}1_{\{k\leq-1\}}]\big)\otimes b_{|k|}^{k}\big\}.
\end{align}
\end{lem}

\begin{proof}
From (\ref{eq:BnGenerating}), by equating the coefficients of $w^{n}$
with eigenvalue $\beta$ with $|\beta|\leq n$,
\begin{align*}
b_{n}^{\beta} & =\sum_{k=1}^{n}\rho_{n-k}^{\beta-k}\left((\mathbf{1}-\mathrm{i}(v_{1}+v_{2})-\frac{1}{2}\left(v_{1}\otimes v_{2}+v_{2}\otimes v_{1}\right))^{-1}\otimes(\mathbf{1}-\mathrm{i}v_{1})\right)\otimes b_{k}^{k}\\
 & +\sum_{k=-1}^{-n}\rho_{n-|k|}^{\beta-k}\left((\mathbf{1}-\mathrm{i}(v_{1}+v_{2})-\frac{1}{2}\left(v_{1}\otimes v_{2}+v_{2}\otimes v_{1}\right))^{-1}\otimes(\mathbf{1}-\mathrm{i}v_{2})\right)]\otimes b_{|k|}^{k}\\
 & +\rho_{n}^{\beta}\left((\mathbf{1}-\mathrm{i}(v_{1}+v_{2})-\frac{1}{2}\left(v_{1}\otimes v_{2}+v_{2}\otimes v_{1}\right))^{-1}\otimes(\mathbf{1}-\mathrm{i}v_{1}-\mathrm{i}v_{2})\right)\otimes b_{0}^{0}.
\end{align*}

Note that $(u+v)^{-1}\otimes u=\mathbf{1}-(u+v)^{-1}\otimes v$ and
so let $u=\mathbf{1}-\mathrm{i}(v_{1}+v_{2})-\frac{1}{2}\left(v_{1}\otimes v_{2}+v_{2}\otimes v_{1}\right)$,
\begin{align*}
 & b_{n}^{\beta}\\
= & \sum_{k=1}^{n}\rho_{n-k}^{\beta-k}\left(\mathbf{1}+u{}^{-1}\otimes(\mathrm{i}v_{2}+\frac{1}{2}(v_{1}\otimes v_{2}+v_{2}\otimes v_{1}))\right)\otimes b_{k}^{k}\\
 & +\sum_{k=-1}^{-n}\rho_{n-|k|}^{\beta-k}\left(\mathbf{1}+u{}^{-1}\otimes(\mathrm{i}v_{1}+\frac{1}{2}(v_{1}\otimes v_{2}+v_{2}\otimes v_{1}))\right)]\otimes b_{|k|}^{k}\\
 & +\rho_{n}^{\beta}\left(\mathbf{1}+u^{-1}\otimes\frac{1}{2}(v_{1}\otimes v_{2}+v_{2}\otimes v_{1})\right)\otimes b_{0}^{0}\\
= & b_{|\beta|}^{\beta}1_{n=|\beta|}+\sum_{k=-n}^{n}\big\{\rho_{n-|k|}^{\beta-k}\big((\mathbf{1}-\mathrm{i}(v_{1}+v_{2})-\frac{1}{2}\left(v_{1}\otimes v_{2}+v_{2}\otimes v_{1}\right))^{-1}\\
 & \otimes[\frac{1}{2}(v_{1}\otimes v_{2}+v_{2}\otimes v_{1})+\mathrm{i}v_{2}1_{\{k\geq1\}}+\mathrm{i}v_{1}1_{\{k\leq-1\}}]\big)\otimes b_{|k|}^{k}\big\}.
\end{align*}
\end{proof}
We now solve for $b_{|k|}^{k}$ using the boundary condition (\ref{eq:BoundaryODE}).
\begin{lem}
Let $u=\mathbf{1}-\mathrm{i}(v_{1}+v_{2})-\frac{1}{2}\left(v_{1}\otimes v_{2}+v_{2}\otimes v_{1}\right)$
and let $q=\frac{1}{2}\left(v_{1}\otimes v_{2}+v_{2}\otimes v_{1}\right)$.
Define a linear map $S:T((\mathbb{R}^{2}))\rightarrow T((\mathbb{R}^{2}))$
by 
\begin{align*}
S(a) & =\sum_{\beta}\sum_{k}\sum_{n\geq|\beta|\vee|k|}\frac{(-1)^{\frac{n-|\beta|}{2}}|\beta|!2^{|\beta|}}{(\frac{n-\beta}{2})!(\frac{n+\beta}{2})!2^{n}}\cdot\\
 & \rho_{n-|k|}^{\beta-k}\big(u^{-1}\otimes[q+\mathrm{i}v_{2}1_{\{k\geq1\}}+\mathrm{i}v_{1}1_{\{k\leq-1\}}]\big)\otimes\rho^{k}(a).
\end{align*}
Then 
\begin{equation}
b_{|k|}^{k}=\rho^{k}\left[\sum_{l=0}^{\infty}(-1)^{l}S^{l}(\mathbf{1})\right],\label{eq:bkkFormula}
\end{equation}
with $S^{l}$ meaning the composition of $S$ with itself $l$-times. 
\end{lem}

\begin{proof}
Note that since 
\[
\phi(r)=\sum_{\beta\in\mathbb{Z}}\sum_{n\geq|\beta|}a_{n}^{\beta}r^{n}=\sum_{\beta\in\mathbb{Z}}\sum_{n\geq|\beta|}\frac{(-1)^{\frac{n-\beta}{2}}}{(\frac{n-\beta}{2})!(\frac{n+\beta}{2})!}b_{n}^{\beta}r^{n}.
\]
Using (\ref{eq:bnFormula}), we have 
\begin{align*}
\phi(r)= & \sum_{n,\beta:n\geq|\beta|}\frac{(-1)^{\frac{n-\beta}{2}}}{(\frac{n-\beta}{2})!(\frac{n+\beta}{2})!}\frac{1}{2^{n}}b_{n}^{\beta}r^{n}\\
= & \sum_{\beta}\frac{1}{|\beta|!}\frac{(-1)^{\frac{|\beta|-\beta}{2}}}{2^{|\beta|}}b_{|\beta|}^{\beta}r^{|\beta|}\\
 & +\sum_{\beta}\sum_{n\geq|\beta|}\frac{(-1)^{\frac{n-\beta}{2}}}{(\frac{n-\beta}{2})!(\frac{n+\beta}{2})!2^{n}}\sum_{k=-n}^{n}\rho_{n-|k|}^{\beta-k}(u^{-1}\otimes[q+\mathrm{i}v_{2}1_{\{k\geq1\}}+\mathrm{i}v_{1}1_{\{k\leq-1\}}])\otimes b_{|k|}^{k}r^{n}.
\end{align*}
Using the boundary condition $\phi(1)=\mathbf{1}$ and projecting
onto $V^{\beta}$, 
\begin{align*}
1_{\beta=0}= & \frac{1}{|\beta|!}\frac{(-1)^{\frac{|\beta|-\beta}{2}}}{2^{|\beta|}}b_{|\beta|}^{\beta}\\
 & +\sum_{n\geq|\beta|}\frac{(-1)^{\frac{n-\beta}{2}}}{(\frac{n-\beta}{2})!(\frac{n+\beta}{2})!2^{n}}\sum_{k=-n}^{n}\rho_{n-|k|}^{\beta-k}\big(u{}^{-1}\otimes[q+\mathrm{i}v_{2}1_{\{k\geq1\}}+\mathrm{i}v_{1}1_{\{k\leq-1\}}]\big)\otimes b_{|k|}^{k}\\
= & \frac{1}{|\beta|!}\frac{(-1)^{\frac{|\beta|-\beta}{2}}}{2^{|\beta|}}b_{|\beta|}^{\beta}\\
 & +\sum_{k}\big[\sum_{n\geq|\beta|\vee|k|}\frac{(-1)^{\frac{n-\beta}{2}}}{(\frac{n-\beta}{2})!(\frac{n+\beta}{2})!2^{n}}\rho_{n-|k|}^{\beta-k}\big(u{}^{-1}\otimes[q+\mathrm{i}v_{2}1_{\{k\geq1\}}+\mathrm{i}v_{1}1_{\{k\leq-1\}}]\big)\big]\otimes b_{|k|}^{k}.
\end{align*}
Therefore 
\begin{align}
1_{\beta=0} & =b_{|\beta|}^{\beta}+\sum_{k}\big[\sum_{n\geq|\beta|\vee|k|}\frac{(-1)^{\frac{n-|\beta|}{2}}|\beta|!2^{|\beta|}}{(\frac{n-\beta}{2})!(\frac{n+\beta}{2})!2^{n}}\rho_{n-|k|}^{\beta-k}\big(u{}^{-1}\otimes[q+\mathrm{i}v_{2}1_{\{k\geq1\}}+\mathrm{i}v_{1}1_{\{k\leq-1\}}]\big)\big]\otimes b_{|k|}^{k}.\label{eq:BoundaryEquation}
\end{align}
Summing (\ref{eq:BoundaryEquation}) over $\beta$, we have 
\[
\mathbf{1}=(I+S)(\sum_{\beta}b_{|\beta|}^{\beta}),
\]
and therefore 
\[
\sum_{k\in\mathbb{Z}}b_{|k|}^{k}=\sum_{l=0}^{\infty}(-1)^{l}S^{l}(\mathbf{1}).
\]
\end{proof}
\begin{thm}
\label{thm:MainTheorem}Let $\Phi_{\mathbb{D}}$ be the expected signature
of $B^{z}$ up to the first exit time of the planar disc $\mathbb{D}$.
If $z=(r\cos\theta,r\sin\theta)^{T}$, then 
\[
\Phi_{\mathbb{D}}(z)=\mathbf{R}(\theta)\big[\sum_{\beta\in\mathbb{Z}}\sum_{n\geq|\beta|}\frac{(-1)^{\frac{n-\beta}{2}}}{(\frac{n-\beta}{2})!(\frac{n+\beta}{2})!}b_{n}^{\beta}r^{n}\big],
\]
with $b_{n}^{\beta}$ being given by (\ref{eq:bnFormula}) and (\ref{eq:bkkFormula}). 
\end{thm}

Although the formula (\ref{eq:bnFormula}) is quite complicated, it
reduces to a simpler formula if we are only looking for the coefficient
of $r^{n}$ in $\rho_{n}(\Phi_{\mathbb{D}}(z))$. 
\begin{cor}
\label{cor:MainCorollary}Let $u=\mathbf{1}-\mathrm{i}(v_{1}+v_{2})-\frac{1}{2}\left(v_{1}\otimes v_{2}+v_{2}\otimes v_{1}\right)$
and let $q=\frac{1}{2}\left(v_{1}\otimes v_{2}+v_{2}\otimes v_{1}\right)$.
We have that
\begin{equation}
\rho_{n}[a_{n}]=\begin{cases}
1, & \text{if }n=0\\
0, & \text{if }n=1\\
\frac{1}{2^{n}}\sum_{m=1}^{n-1}\frac{(-1)^{m}}{m!(n-m)!}\rho_{n-2}^{n-2m}[u^{-1}]\otimes q & \text{if }n\geq2.
\end{cases}\label{eq:TopDegree}
\end{equation}
\end{cor}

\begin{proof}
We need to first calculate
\begin{align}
\rho_{n}[b_{n}^{\beta}] & =\rho_{n}\left[b_{|\beta|}^{\beta}1_{n=|\beta|}+\sum_{k=-n}^{n}\rho_{n-|k|}^{\beta-k}(u^{-1}\otimes[q+\mathrm{i}v_{2}1_{\{k\geq1\}}+\mathrm{i}v_{1}1_{\{k\leq-1\}}])\otimes b_{|k|}^{k}\right]\nonumber \\
 & =\rho_{|\beta|}[b_{|\beta|}^{\beta}]1_{n=|\beta|}+\sum_{k=-n}^{n}\rho_{n-|k|}^{\beta-k}(u^{-1}\otimes[q+\mathrm{i}v_{2}1_{\{k\geq1\}}+\mathrm{i}v_{1}1_{\{k\leq-1\}}])\otimes\rho_{|k|}[b_{|k|}^{k}].\label{eq:bnLowestDegree}
\end{align}
Note that $\rho_{|k|}[b_{|k|}^{k}]$ has tensor degree $|k|$ and
eigenvalue $k$, this means for some $c_{k}\in\mathbb{R}$
\[
b_{|k|}^{k}=\begin{cases}
c_{k}v_{1}^{\otimes k}, & \text{if }k>0\\
c_{k}v_{2}^{\otimes k}, & \text{if }k<0\\
c_{k}, & \text{if }k=0.
\end{cases}
\]
Now because $S(\mathbf{1})$ is of the form $\alpha\otimes\frac{1}{2}(v_{1}\otimes v_{2}+v_{2}\otimes v_{1})$,
so is $S^{l}(\mathbf{1})$ for all $l\geq1$ and therefore 
\[
\rho_{|k|}^{k}[S^{l}(\mathbf{1})]=0\qquad\forall k\in\mathbb{Z},l\geq1.
\]
That means 
\[
\rho_{|k|}[b_{|k|}^{k}]=\begin{cases}
\mathbf{1}, & \text{if }k=0\\
0, & \text{if }k\neq0.
\end{cases}
\]
Substituting into (\ref{eq:bnLowestDegree}) gives 
\[
\rho_{n}[b_{n}^{\beta}]=1_{n=\beta=0}+\rho_{n}^{\beta}\big(u{}^{-1}\otimes\frac{1}{2}(v_{1}\otimes v_{2}+v_{2}\otimes v_{1})\big).
\]
In particular, 
\[
\rho_{n}[a_{n}]=1_{n=0}+\sum_{\beta=-n}^{n}\frac{(-1)^{\frac{n-\beta}{2}}}{(\frac{n-\beta}{2})!(\frac{n+\beta}{2})!2^{n}}\rho_{n}^{\beta}\big(u{}^{-1}\otimes\frac{1}{2}(v_{1}\otimes v_{2}+v_{2}\otimes v_{1})\big).
\]
Note that for any $a\in T((\mathbb{R}^{2}))$, $\rho_{n}^{\beta}(a)=0$
whenever $n-\beta$ is odd. Therefore, we may substitute $m=\frac{n-\beta}{2}$
to obtain
\begin{align*}
\rho_{n}[a_{n}] & =1_{n=0}+\sum_{m=0}^{n}\frac{(-1)^{m}}{m!(n-m)!2^{n}}\rho_{n}^{n-2m}\big(u{}^{-1}\otimes\frac{1}{2}(v_{1}\otimes v_{2}+v_{2}\otimes v_{1})\big)\\
 & =1_{n=0}+\sum_{m=0}^{n}\frac{(-1)^{m}}{m!(n-m)!2^{n}}\rho_{n-2}^{n-2m}\big(u{}^{-1}\big)\otimes\frac{1}{2}(v_{1}\otimes v_{2}+v_{2}\otimes v_{1})
\end{align*}
Note that for any $a\in T((\mathbb{R}^{2}))$, $\rho_{n}^{\beta}(a)=0$
for $|\beta|>n$, therefore the terms corresponding to $m=0$ ad $m=n$
vanishes and we obtain (\ref{eq:TopDegree}).
\end{proof}
\begin{example}
We now describe how to use Corollary \ref{cor:MainCorollary} to find
the terms in $\rho_{n}[\Phi_{\mathbb{D}}(z)]$ with polynomial degree
$4$ in $z$. By Corollary \ref{cor:MainCorollary}, this is given
by
\begin{align*}
E_{4,4} & =\frac{r^{4}}{2^{4}}\mathbf{R}(\theta)\left\{ \sum_{m=1}^{3}\frac{(-1)^{m}}{m!(4-m)!}\rho_{2}^{4-2m}[\left(\mathbf{1}-\mathrm{i}(v_{1}+v_{2})-\frac{1}{2}\left(v_{1}\otimes v_{2}+v_{2}\otimes v_{1}\right)\right)^{-1}]\right.\\
 & \left.\otimes\frac{1}{2}\left(v_{1}\otimes v_{2}+v_{2}\otimes v_{1}\right)\right\} .
\end{align*}
Note that 
\begin{align*}
 & \rho_{2}^{4-2m}[\left(\mathbf{1}-\mathrm{i}(v_{1}+v_{2})-\frac{1}{2}\left(v_{1}\otimes v_{2}+v_{2}\otimes v_{1}\right)\right)^{-1}]\\
= & \rho_{2}^{4-2m}[\sum_{l=0}^{\infty}\left[\mathrm{i}(v_{1}+v_{2})+\frac{1}{2}\left(v_{1}\otimes v_{2}+v_{2}\otimes v_{1}\right)\right]^{\otimes l}]\\
= & \rho^{4-2m}[\left[\mathrm{i}(v_{1}+v_{2})\right]^{\otimes2}+\frac{1}{2}\left(v_{1}\otimes v_{2}+v_{2}\otimes v_{1}\right)]\quad\text{(projecting to degree 2 terms)}\\
= & \begin{cases}
(\mathrm{i}v_{1})^{\otimes2}, & m=1\\
-\frac{1}{2}\left(v_{1}\otimes v_{2}+v_{2}\otimes v_{1}\right) & m=2\\
(\mathrm{i}v_{2})^{\otimes2}, & m=3.
\end{cases}
\end{align*}
For the final line above, we note that $\rho^{k}$ is a projection
onto terms where the number of $v_{1}$ minus the number of $v_{2}$
is equal to $k$. 

Substituting $v_{1}=\mathrm{i}e_{1}+e_{2}$, $v_{2}=-\mathrm{i}e_{1}+e_{2}$
and $\frac{1}{2}(v_{1}\otimes v_{2}+v_{2}\otimes v_{1})=e_{1}^{\otimes2}+e_{2}^{\otimes2}$,
\begin{align*}
E_{4,4} & =\frac{r^{4}}{2^{4}}\mathbf{R}(\theta)\left[\big(-\frac{1}{3!}(-e_{1}+\mathrm{i}e_{2})^{\otimes2}-\frac{1}{2!2!}(e_{1}^{\otimes2}+e_{2}^{\otimes2})-\frac{1}{3!}(e_{1}+\mathrm{i}e_{2})^{\otimes2}\big)\otimes(e_{1}^{\otimes2}+e_{2}^{\otimes2})\right]\\
 & =\frac{r^{4}}{2^{4}}\mathbf{R}(\theta)\left[-\frac{7}{12}e_{1}^{\otimes2}+\frac{1}{12}e_{2}^{\otimes2}\right]\otimes(e_{1}^{\otimes2}+e_{2}^{\otimes2}).
\end{align*}
Note that 
\begin{align*}
\mathbf{R}(\theta)[e_{1}] & =\cos\theta e_{1}+\sin\theta e_{2},\mathbf{R}(\theta)[e_{2}]=-\sin\theta e_{1}+\cos\theta e_{2}\\
 & \mathbf{R}(\theta)[e_{1}^{\otimes2}+e_{2}^{\otimes2}]=e_{1}^{\otimes2}+e_{2}^{\otimes2}.
\end{align*}
Therefore putting in $z_{1}=r\cos\theta$ and $z_{2}=r\sin\theta$,
\begin{align*}
E_{4,4} & =\frac{1}{2^{4}}\left[-\frac{7}{12}(z_{1}e_{1}+z_{2}e_{2})^{\otimes2}+\frac{1}{12}(-z_{2}e_{1}+z_{1}e_{2})^{\otimes2}\right]\otimes[(z_{1}e_{1}+z_{2}e_{2})^{\otimes2}+(-z_{2}e_{1}+z_{1}e_{2})^{\otimes2}]\\
 & =\frac{(z_{1}^{2}+z_{2}^{2})}{192}\left[-7(z_{1}e_{1}+z_{2}e_{2})^{\otimes2}+(-z_{2}e_{1}+z_{1}e_{2})^{\otimes2}\right]\otimes[e_{1}^{\otimes2}+e_{2}^{\otimes2}],
\end{align*}
which is consistent with that calculated as part of Remark 3.5 in
\cite{LyonsNi}.
\end{example}

\end{document}